\documentclass[11pt]{article}
\usepackage{color,latexsym,amsthm,amsmath,amssymb,path,enumitem,url}
\usepackage[utf8]{inputenc}
\usepackage[english]{babel}
\newcommand{\ignore}[1]{}

\newtheorem{theorem}{Theorem}
\newtheorem{lemma}[theorem]{Lemma}

\newtheorem{proposition}[theorem]
{Proposition}
\newtheorem*{theorem*}{Theorem}
\newtheorem*{corollary*}{Corollary}

\newcommand{\Proof}[1]
        {
        \noindent
        \emph{Proof #1.}~
        }
\newsavebox{\smallProofsym}                     % smallproofsym.tex
\savebox{\smallProofsym}                        %
        {%                                      %
        \begin{picture}(7,7)                    %
        \put(0,0){\framebox(6,6){}}             %
        \put(0,2){\framebox(4,4){}}             %
        \end{picture}                           %
        }                                       %
\newcommand{\smalleop}[1]
        {
        \mbox{} \hfill #1~~\usebox{\smallProofsym}\!\!\!\!\!\!\
        }

\usepackage{graphicx,psfrag}
\usepackage[rflt]{floatflt}

\usepackage{dsfont}

\def\P{\mathcal{P}}
\def\Q{\mathcal{Q}}

\begin{document}
\pagenumbering{arabic}

\title{
The bounds for the number of
linear extensions\\
via chain and antichain coverings}

\author{I. A. Bochkov\thanks{St. Petersburg State University, {\sl bycha@yandex.ru}}, 
F. V. Petrov\thanks{St. Petersburg State University, St. Petersburg Department of V. A. Steklov Mathematical
Institute RAS,
{\sl f.v.petrov@spbu.ru}. The work was supported by the Foundation for the
Advancement of Theoretical Physics and Mathematics ``BASIS''.}}

\maketitle

\begin{abstract}
Let $(\P,\leqslant)$ be a finite poset.
Define the numbers $a_1,a_2,\ldots$ 
(respectively, $c_1,c_2,\ldots$)
so that $a_1+\ldots+a_k$ 
(respectively, $c_1+\ldots+c_k$)
is the maximal number
of elements of $\P$
which may be covered by 
$k$ antichains (respectively,
 $k$ chains.)
 Then the number
$e(\P)$ of linear extensions of 
poset $\P$ is not less than
 $\prod a_i!$ and not more
 than $n!/\prod c_i!$.
 A corollary: if 
 $\P$ is partitioned onto
 disjoint antichains of sizes
 $b_1,b_2,
\ldots$, then $e(\P)\geqslant \prod b_i!$.
\end{abstract}

\section{Introduction}
Let $(\P,\leqslant)$ be a finite poset. 

In a recent paper \cite{MPP}
the following double inequality
for the number $e(\P)$  of linear extensions of
 $\P$ is applied. Partition
 $\P$ onto disjoint antichains
$\P=A_1\sqcup A_2\sqcup A_3\ldots,$ 
where $A_i$ is the antichain
of elements with rank $i$. 
Also partition $\P$ in arbitrary way
onto disjoint chains $\P=C_1\sqcup C_2\sqcup C_3\ldots$.
 Then
\begin{equation}\label{MPP_e}
   \frac{n!}{\prod |C_i|!}\geqslant e(\P)
\geqslant \prod |A_i|! 
\end{equation}
To quote \cite{MPP}:
``These bounds are probably folklore; for the lower bound see e.g
\cite{Br}.''

We prove that the right inequality
in \eqref{MPP_e} holds for
arbitrary antichain partition. 
We also improve both inequalities
in terms of Greene--Kleitman--Fomin parameters,
which we define now.

The  \emph{antichain Greene--Kleitman--Fomin parameters
 $a_1\geqslant a_2\geqslant \ldots$
 }
 of a finite poset $\P$
 are defined as follows:
$a_1+\ldots+a_k$ is the maximal number
of elements of $\P$
which may be covered by 
$k$ antichains ($k=1,2,\ldots$). 
The fact that the sequence
$(a_i)$ is weakly decreasing 
is a part of Greene--Kleitman--Fomin
theorem \cite{2,3}. Another claim of this
theorem is that for the partition 
$n=c_1+c_2+\ldots$ conjugate
to the partition $n=a_1+a_2+\ldots$
the sum $c_1+\ldots+c_k$ 
is the maximal number
of elements of  $\P$
which may be covered by
$k$ chains. The numbers 
$c_1,c_2,\ldots$ are called 
\emph{chain Greene--Kleitman--Fomin  parameters} of poset
$\P$.

The main result of this paper is

\begin{theorem}
\begin{equation}\label{aandc}
\frac{n!}{\prod c_i!}
\geqslant e(P)\geqslant \prod_i a_i!
\end{equation}
\end{theorem}

\section{Majorization lemmata}

Let $X$ be a finite multiset
consisting of non-negative numbers.
For non-negative integer $k$ define
$s_k(X)$ as the sum of $\min(k,|X|)$
maximal elements of 
multiset  $X$; also
denote $s(X)=s_{|X|}(X)$
the sum of all elements of $X$. 
We say that multiset $X$ \emph{majorizes}
another multiset $Y$ of non-negative numbers and write
$X\succ Y$ if $s(X)=s(Y)$ and $s_k(X)\geqslant s_k(Y)$
for all $k$.

We need the following version of Karamata's majorization inequality.

\begin{theorem*}[Karamata inequality]
Let $f(x):\{0,1,\ldots\}\to (0,+\infty)$
be a log-convex function,
i.e.,
$f(0)=1$ and
$f(x+1)/f(x)$ is an increasing function
of $x$. Next, let $X,Y\subset \{0,1,\ldots\}$ be two finite multisets.
Then $X$ majorizes $Y$
if and only if
 $\prod_X f\geqslant \prod_Y f$ 
 for any log-convex function $f$.
\end{theorem*}

This theorem immediately yields
\begin{proposition}\label{opr}
For multisets $X,Y,Z\subset \{0,1,\ldots\}$
the conditions
$X\succ Y$ and
$X\cup Z\succ Y\cup Z$ 
are equivalent.
\end{proposition}

Further we use Karamata inequality
for the log-convex function $f(x)=x!$

We also use the following simple fact

\begin{proposition}\label{predl_maj}
Let $X,Y$ be two
finite multisets consisting of
non-negative integers,
and the sum of elements in $Y$
is by 1 less than the sum of elements
in $X$:
$s(Y)=s(X)-1$. Denote by
$x_i$ and $y_i$ there $i$-th
largest elements, respectively.
Also denote
$x_i=0$ for $i>|X|$, 
$y_i=0$ for $i>|Y|$. (In these
notations we have
 $s_k(X)=\sum_{i\leqslant k} x_i$,
$s_k(Y)=\sum_{i\leqslant k} y_i$.)

Assume that $s_k(X)\geqslant s_k(Y)\geqslant s_k(X)-1$
for all $k$, and there
exists 
$m$ such that $s_k(X)=s_k(Y)$
for all $k=1,2,\ldots,m-1$. 
Then $x_m>0$ and
$Y\succ (X\cup \{x_m-1\})\setminus x_m$.
\end{proposition}

\begin{proof}
We get $y_i=x_i$ for $i=1,2,\ldots,m-1$.
Consider the
minimal $j$ such that
$y_j\ne x_j$. Then $j\geqslant m$
and denoting $\varepsilon=s_j(X)-s_j(Y)\in \{0,1\}$ we have
$$
y_j=s_j(Y)-s_{j-1}(Y)=
s_j(Y)-s_{j-1}(X)=S_j(X)-\varepsilon-s_{j-1}(X)
=x_j-\varepsilon,
$$
therefore $\varepsilon=1$ and $y_j=x_j-1$.
So we get $x_m\geqslant x_j=y_j+1>0$.

Assume that $x_{j+1}=x_j$.
Then $y_{j+1}\leqslant y_j=x_j-1=x_{j+1}-1$
and $s_{j+1}(Y)\leqslant s_{j+1}(X)-2$,
a contradiction. Therefore
$x_{j+1}\leqslant x_j-1$,
and $Y$ majorizes multiset
$X_j:=(X\cup \{x_j-1\})\setminus x_j$:
indeed, $s_k(Y)=s_k(X_j)=s_k(X)$
for $k\leqslant j-1$, and
$s_k(Y)\geqslant s_k(X)-1=s_k(X_j)$
for $k\geqslant j$ .

Since $j\geqslant m$, we have
$\{x_j-1,x_m\}\succ \{x_m-1,x_{j}\}$,
and by Proposition
\ref{opr} we get
$$Y\succ X_j=(X\setminus \{x_m,x_j\})\cup 
\{x_j-1,x_{m}\}\succ 
(X\setminus \{x_m,x_j\})\cup 
\{x_m-1,x_{j}\}=X_m,$$
as needed.
\end{proof}

\section{Lower bound}

Let $(\P,\leqslant)$ be a finite poset, $A\subset \P$ be an antichain.
The proof of the main
result of
\cite{1} implies the inequality
\begin{equation}\label{step}
e(\P)\geqslant \sum_{x\in A} e(\P-x)
\end{equation}
for the number of
linear extensions $e(\cdot)$.
Inequality \eqref{step}
is also proved differently in \cite{S}.
In \cite{1} it was proved that  \eqref{step} turns into equality if 
antichain $A$
have non-empty intersection with
any maximal chain.

For sake of completeness we prove \eqref{step}.
Define an injection from linear extensions
of posets
$\P-x$, where $x\in A$, 
to linear extensions
of poset $\P$.
A linear extension of
$\P-x$, $x\in \P$,
is understood as
an order-preserving bijection
$f:\P-x\mapsto \{2,3,\ldots,n\}$,
where $n=|\P|$. For such
$f$ we construct an order-preserving bijection
$\phi f:\P\mapsto \{1,2,\ldots,n\}$
as follows. Consider the
\emph{greedy falling chain from $x$}:
put $x_0=x$, define $x_{i+1}$ 
as an element of $\P$ which
is less than
 $x_i$ for which $f(x_i)$
 is maximal possible. We get a chain $x_0>x_1>\ldots>x_t$.
 
Let's shift the values along this chain. Namely, put
$\phi f (x_i)=f(x_{i+1})$ for $i=0,1,\ldots,t-1$,
$\phi f(x_t)=1$. For $y\notin \{x_0,\ldots,x_t\}$
put $\phi f (y)=f(y)$. 
Note that $\phi f$
is order-preserving 
due to the greedy property.
Indeed, if $y\{x_0,\ldots,x_t\}$ and $x_i>y$, then by greediness we get
 $f(y)\leqslant f(x_{i+1})=\phi f(y)$,
 other inequalities are obvious.

It is straightforward that
for the map $\phi f$ the chain
$x_t<x_{t-1}<\ldots<x_0$ is the 
\emph{greedy increasing chain}:
$x_t=(\phi f)^{-1}(1)$, 
and each next element realizes the
minimal possible value of $\phi f$.

Therefore, if the maps
$f:\P-x\mapsto \{2,3,\ldots,n\}$
and $g:\P-y\mapsto \{2,3,\ldots,n\}$ 
satisfy
$\phi f=\phi g$, then $x,y$ 
belong to the greedy increasing
chain of $\phi f$. If 
$x,y$ also belong to the antichain
 $A$, then we get $x=y$ and $f=g$. 

Injectivity of
$\varphi$ is proved, it implies
inequality \eqref{step}.

\begin{proof}[Proof
of lower bound in \eqref{aandc}]
Induction on $n=|\P|$. Base
 $n=1$ is obvious.
 
 Step from $n-1$ to $n$. 

Let $A$ be a maximal
antichain in $\P$, then
$|A|=a_1$. Fix $x\in A$.
Let $r_1\geqslant r_2\ldots$ 
denote antichain Greene--Kleitman--Fomin
parameters for
$\P-x$. Due to Proposition \ref{predl_maj} the multiset
$\{r_1,r_2,\ldots\}$ majorizes
the multiset 
$\{a_1-1,a_2,a_3,\ldots\}$.
By Karamata inequality for log-convex function
$f(x)=x!$ we get
$$\prod r_i!\geqslant (a_1-1)!\prod_{i>1} a_i!=
\frac1{a_1}\prod_i a_i!$$
Therefore by induction
proposition we have
 $e(\P-x)\geqslant \frac1{a_1}\prod_i a_i!$
For all $x\in A$. 
Summing up over all
  $x\in A$ and using
 \eqref{step} we see
 that indeed 
$e(\P)\geqslant \prod_i a_i!$, 
as needed.
\end{proof}

\begin{corollary*}
Let $\P$ be partitioned onto
antichains of sizes
$c_1,c_2,\ldots$. Then
$e(\P)\geqslant \prod_i c_i!$.
\end{corollary*}

\begin{proof}
It suffices to note that multiset $\{c_1,c_2,\ldots\}$ is
majorized by the multiset $\{a_1,a_2,\ldots\}$
and apply Karamata inequality for factorial. Alternatively,
we may prove this claim directly
by induction using \eqref{step}.
\end{proof}

\section{Upper bound}

For bounding the number
of linear extensions from above we need

\begin{lemma}\label{st}
Let $\P$ be a finite poset, 
$A\subset \P$ be the antichain of maximal
elements in $\P$,
$c_1\geqslant c_2\geqslant\ldots $ 
be the chain Greene--Kleitman--Fomin
parameters of $\P$. Then the elements of
$A$ may be enumerated as
 $x_1,x_2,\ldots,x_{|A|}$ so that
 for all $i=1,2,\ldots,|A|$ 
 there exist $i$ chains whose
 maximal elements belong to
 the set
$\{x_1,\ldots,x_i\}\cup (\P\setminus A)$ with total size
$c_1+\ldots+c_i$.
\end{lemma}

\begin{proof}
Assume that the elements $x_1,\ldots,x_{i}$
are already chosen and satisfy the conditions
of Lemma. Contract the set 
$A\setminus \{x_1,\ldots,x_i\}$ to a new one element
 $t$, denote the new poset $\Q$. 
 The first $i$ chain Greene--Kleitman--Fomin
parameters of $\Q$ are the same as for $\P$.
Denote the
$(i+1)$-th parameter by $\alpha$. We should
prove that $\alpha=c_{i+1}$:
it allows to choose appropriate $x_{i+1}$.
Assume that on the contrary $\alpha<c_{i+1}$. Then
by Greene--Kleitman--Fomin duality
we may find $i$ chains in
$\Q$ of total size $c_1+\ldots +c_{i}$
(not containing $t$) and $\alpha$
antichains so that they cover all
elements of $\Q$ and each chain
intersects each antichain. Note that the antichain
containing 
 $t$ remains an antichain
 if we replace $\Q$ back to $\P$
 (by splitting the element $t$). 
 But then $i+1$ chains in $\P$
 may cover at most $c_1+\ldots+c_i+\alpha$
 elements: at most $\alpha(i+1)$
 elements may be covered by $\alpha$ antichains, and
 exactly $c_1+\ldots+c_i-i\alpha$ elements remain.
 A contradiction.
\end{proof}

\begin{proof}[Proof of the upper bound in \eqref{aandc}]
By induction, we suppose
the inequality proved for 
posets with $n-1$ elements
(the base $n=1$ is obvious).
Enumerate antichain $A$
of maximal elements
as in Lemma \ref{st}. 
Let
$r_1\geqslant r_2\geqslant \ldots$ be chain 
Greene--Kleitman--Fomin
parameters of the poset $\P-x_i$.
It is clear that for all
$j$ we have $c_1+\ldots+c_j\geqslant r_1+\ldots+r_j\geqslant
c_1+\ldots+c_j-1$, and for $j\leqslant i-1$
the equality $r_1+\ldots+r_j=
c_1+\ldots+c_j$ holds. 
Therefore using Proposition
\ref{predl_maj} we conclude 
that the multiset $\{r_1,r_2,\ldots\}$
majorizes the multiset $\{c_1,c_2,\ldots,c_{j-1},c_j-1,c_{j+1},\ldots\}$.
By Karamata inequality for factorial we have $\prod r_j!\geqslant \frac1{c_j}\prod_j c_j!$,
and using induction proposition
we get
$$
e(\P-x_i)\leqslant \frac{(n-1)!}{\prod r_j!}
\leqslant \frac{c_j(n-1)!}{\prod c_j!}.
$$
Sum up this by all $i$ 
and apply the obvious
equality $e(\P)=\sum_i e(\P-x_i)$
we complete the induction step.
\end{proof}

\section{Accuracy of the
bounds}

The upper and lower bounds
in \eqref{aandc} are
close enough: there ratio is always $e^{O(n\log\log n)}=n!^{o(1)}$ (but alas
worse than exponential in $n$).
This follows from the following
general inequality.

\begin{theorem}
For any two conjugate partitions
$a_1\geqslant a_2\geqslant\ldots$
and $c_1\geqslant c_2\geqslant\ldots$
of the positive integer $n$ 
onto positive integer parts the inequality
$$\prod_i a_i! \prod_i c_i!\geqslant \left(\frac{n}{eH_n}\right)^n=n!e^{-n\log \log n+O(\log n)}$$
holds (here $H_n=1+1/2+\ldots+1/n=\log n+O(1)$ 
is a Harmonic sum).
\end{theorem}

\begin{proof}
We have $$\prod_i c_i!=\prod_k k^{|i:c_i\geqslant k|}=\prod_k k^{a_k},$$
thus the inequality to prove is
\begin{equation}\label{nervo}
    \prod_k k^{a_k}a_k!\geqslant \left(\frac{n}{eH_n}\right)^n,
\end{equation}
which we now prove for
arbitrary non-negative integers
$a_1,a_2,\ldots,a_n$ which sum up to $n$.
Without loss of generality $a_1,a_2,\ldots,a_n$
are chosen so that the left hand
side of \eqref{nervo} is minimal possible.
Choose two indices $k,\ell$ such that
 $a_\ell>0$ and try to replace
 $a_k$ to $a_k+1$, $a_\ell$ to $a_\ell-1$.
 Left hand side of \eqref{nervo} is multiplied by  $k(a_k+1)/(\ell a_\ell)$, thus
 $k(a_k+1)\geqslant \ell a_\ell$. This last
 inequality holds for $a_\ell=0$ too. 
 Sum up the inequalities
  $k(a_k+1)\cdot \frac{1}\ell\geqslant a_\ell$ over $\ell=1,2,\ldots,n$ we get
 $k(a_k+1)H_n\geqslant a_1+\ldots+a_n=n$.
 Next, using the inequality
 $a!\geqslant (\frac{a+1}e)^a$
 which holds for all non-negative integer
 $a$ (it may be proved by induction, for example)
 we get
 $$
 \prod_k k^{a_k}a_k!\geqslant \prod_k \left(\frac{k(a_k+1)}e\right)^{a_k}
 \geqslant \prod_k \left(\frac{n}{eH_n}\right)^{a_k}=
 \left(\frac{n}{eH_n}\right)^n,
 $$
 as needed.
\end{proof}


\begin{thebibliography}{00}

\bibitem{2} C. Greene, D. J. Kleitman. The structure of Sperner k-families. J. Combin. Th. A, 20 (1976), pp.41-68.


\bibitem{3}
S. V. Fomin, Finite partially ordered sets and Young tableaux, Soviet Math. Dokl. 19
(1978), 1510-1514.
%С. В. Фомин. Конечные частично упорядоченные множества и диаграммы Юнга. Докл. АН СССР, 243(5):1144-1147 (1978) 

\bibitem{1} %DOI: 10.1007/BF00341632
Paul H. Edelman, Takayuki Hibi, Richard Stanley.
A recurrence for linear extensions.
Order 6(1989), no. 1, pp. 15-18.



\bibitem{S} A. Sidorenko.
Inequalities for the number of 
linear extensions. Order 8 (1992), no. 4, 
pp. 331-340.

\bibitem{Br} G.R. Brightwell. 
The number of linear extensions of ranked posets. 
LSE CDAM Res. Report 18 (2003), p. 6.


\bibitem{MPP} A. H. Morales,
I. Pak, G. Panova. Asymptotics of the number of standard Young
tableaux of skew shape.
European Journal of Combinatorics 70 (2018),
pp. 26-49.
\end{thebibliography}
\end{document}